%% file: main.tex
\documentclass[preprint]{imsart}

\usepackage{amsthm,amsmath,amsfonts,graphicx,amssymb}
\usepackage{multirow}
\usepackage{hyperref}

 \RequirePackage[numbers]{natbib}

\firstpage{0}
\lastpage{0}

\startlocaldefs
\newtheorem{defi}{Definition}

\newtheorem{theorem}{Theorem}

\newtheorem{lemma}{Lemma}
\newtheorem{prop}{Proposition}

\usepackage{xspace}

\input{preamble.tex}
\usepackage{hyperref}
\usepackage{graphicx}				

\begin{document}

\begin{frontmatter}

\title{Prediction regions through Inverse Regression
\thanksref{t1}}
\thankstext{t1}{This is an original survey paper}
\runtitle{Prediction regions  through Inverse Regression
}


\author{\fnms{Emilie} \snm{Devijver}\ead[label=e1]{emilie.devijver@kuleuven.be}}
\address{Univ. Grenoble Alpes, CNRS, Grenoble INP\footnote{Institute of Engineering Univ. Grenoble Alpes}, LIG, 38000 Grenoble, France
\\ \printead{e1}}  
\and
\author{\fnms{Emeline} \snm{Perthame}\ead[label=e2]{emeline.perthame@pasteur.fr}}
\address{Institut Pasteur - Bioinformatics and Biostatistics Hub - C3BI, USR 3756 IP CNRS - Paris, France \\ \printead{e2}}

\runauthor{E. Devijver and E. Perthame}

\begin{abstract}
Predict a new response from a covariate is a challenging task in regression, which raises new question since the era of high-dimensional data.
In this paper, we are interested in the inverse regression method from a theoretical viewpoint. 
Theoretical results have already been derived for the well-known linear model, but recently, the curse of dimensionality has increased the interest of practitioners and theoreticians into generalization of those results for various estimators, calibrated for the high-dimension context.
To deal with high-dimensional data, inverse regression is used in this paper. It is known to be  a reliable and efficient approach  when the number of features exceeds the number of observations. Indeed, under some conditions, dealing with the inverse regression problem associated to a forward regression problem drastically reduces the number of parameters to estimate and make the problem tractable.  
When both the responses and the covariates are multivariate, estimators constructed by the inverse regression are studied in this paper, the main result being explicit asymptotic prediction regions for the response. The performances of the proposed estimators and prediction regions are also analyzed through a simulation study and compared with usual estimators.
\end{abstract}


\end{frontmatter}

\section{Introduction}
In a multiple (several response variables) and multivariate (several predictors) regression  framework, one wants to linearly describe a response $\Yb \in \mathbb{R}^L$ from regressors $\Xb \in \mathbb{R}^D$. The standard Gaussian linear model assumes that there exists $\Ab^\star \in \mathbb{R}^{L\times D}$ such that 
\begin{align}
    \Yb = \Ab^\star \Xb + \boldsymbol{\varepsilon} \label{eq:mod_depart}
\end{align} 
where the unobserved error term $\boldsymbol{\varepsilon} \sim \mathcal{N}_L(\mathbf{0}, \Sigmab^\star)$ is a Gaussian white noise.

\vspace{1em}
When considering a high number of predictors, the number of parameters could be quickly larger than the sample size, making the estimates impossible to compute in practice or/and providing bad performances for estimators such as lack of stability. This phenomena is generally referred as curse of dimensionality.
Several tricks have been proposed in the literature to cope with this issue.

One of the most famous method is variable selection based on regularized regression, which reduces the dimension of the regression problem to the subset of the most relevant features. Methods include the Lasso \cite{Tibshirani}, the Dantzig selector \cite{CandesTao}, or the ridge estimator \cite{HoerlKennard} to refer to the most popular. These widely used methods are designed to account for univariate response and few implementations exist for multivariate response,  considering then independent response terms.

Another way to deal with high dimensional data consists in dimension reduction techniques which extract components or latent variables that summarize the information of a large dataset into a small dimension space. For example, the Principal Component Regression (PCR) selects a subset of principal components for regression and  focuses on hyperplanes; the Partial Least Square regression (PLS) projects the predicted variables and looks for latent variables, correlated to both response and covariates,  in order to perform the regression of $\Yb$ on $\Xb$ in a space of lower dimension than $D$ ; and the Sliced Inverse Regression (SIR) introduced in \cite{L91} restricts the regressors to few projections by inverting the role of predictors and response. SIR is based on a prior linear dimension reduction by considering the covariance matrix of the inverse expectation $\mathbb E (\Xv | \Yv)$ (hence the name of the method). The eigenvectors of this covariance matrix are computed in order to find a subspace that retains the information on $\Yv$ contained by the predictors. 
However, the number of axes to retain must be specified beforehand, which is one of the main drawbacks of those methods.  Even if procedures have been proposed to choose this parameter, the results are still sensitive to this choice.

More precisely, in the context of regression with random predictors, several authors proposed reduction dimension techniques based on the joint distribution of both predictors and response \cite{GO96,H92,HA94} to identify components used to reduce the dimension of predictors matrix. Interestingly, while the regression of interest (referred as {\it forward} regression in the literature) usually models the conditional distribution of response given predictors $\Yv | \Xv$, some authors explored the properties of inverse models, meaning that the conditional distribution of predictors is studied given the response $\Xv | \Yv$ (referred as {\it inverse} regression, \cite{O91}). See \cite{C07} for an interesting overview of these techniques. 
 The goal of inverse regression techniques is to  preserve the information on the regression of interest by studying the inverse conditional distribution as it is directly related to the forward conditional distribution of interest. It  consists in inverting the role of response and covariates in the regression model to estimate parameters, taking benefit of the large number of regressors as observations and of the small size of the response.  Note that this inversion regression approach has been studied to estimate Gaussian mixtures of regression models and applied to various data (planetology and spectra \cite{DFH14, PFD16}).

Whereas variable selection methods are mainly used for high-dimensional data, the inverse regression approach is particularly interesting in three specific frameworks. First, when $D>>N$, if a large number of covariates is known to have an impact on the response (e.g. in planetology \cite{DFH14}), selecting variables is not relevant while inverse regression is effective. 
Secondly, when dealing with large dimension for both sample size and number of predictors ($N$ and $D$ large), inverse regression is also a performing method under some weak assumptions:  it avoids the inversion of a large empirical covariance matrix which is time consuming in practice even if it is invertible in theory. Thirdly, inverse regression has the advantage to allow multiple response potentially correlated, which is more and more frequent with real data (e.g. in biology with measurement of multiple phenotypes \cite{sclerose}).

\vspace{1em}

In this paper, we propose to address the multiple linear regression problem of Equation~\eqref{eq:mod_depart} under an inverse regression approach. We study  first the theoretical properties of the estimators of the inverse regression model. Then we focus on a prediction purpose by deriving prediction regions. Indeed, under the linear modeling framework, one can predict a new response from a new covariate using the estimator of regression coefficient matrix $\Ab^\star$. Provided that an estimator of $\Ab^\star$ is available, it is relevant to quantify uncertainty around this prediction. This paper focuses on both confidence region for parameters estimates and prediction regions in high-dimensional settings. 

 Note that few theoretical  confidence intervals have been derived in high dimensional context.
For Lasso based estimators,  \cite{JM14, vDGDBRD14, ZZ14} derive confidence regions for slope coefficient and statistical testing of sparsity for linear model using several tools: relaxed projection \cite{ZZ14}, desparsifying Lasso \cite{vDGDBRD14} or through the computation of an approximate inverse of the Gram matrix \cite{JM14}.
Since those pioneer works, several articles provide extensions for more general models or estimators, as generalised linear model (\cite{vDGDBRD14} for convex loss function, \cite{JvDG15} for subdifferential loss). We also refer to \cite{M15} for groups of variables and \cite{SvdG} for linear regression models with structured sparsity, among others. 
However, those results rely on strong assumptions on the design and although some authors consider more practical aspects \cite{Chao,Lee}, those results still remain difficult to be implemented.

In this paper, we propose to address the linear regression problem of Equation~\eqref{eq:mod_depart} by considering an inverse regression approach rather than sparse regression. We assume that the residuals of the inverse model are independent which reduce the number of parameters to estimate and overcome the dimensionality burden. In this modelling context, assessing confidence in predicted values is one major goal as deriving prediction regions is classical in regression, for the least square estimator for example.  However, when the number of predictors becomes too large, least square method suffers from the curse of dimensionality, has bad performances and is computationally intensive while inverse regression approach tackles this problem. Considering this approach, we get asymptotic and non asymptotic distribution for parameters estimates, and then derive confidence regions for slope coefficients. Moreover, we derive asymptotic prediction regions which quantify uncertainty with prediction through an asymptotic normality theorem. Then, the properties of parameters estimates are illustrated in an intensive simulation study through finite distance examples.

\vspace{1em}

The paper is organised as follows. In Section~\ref{model}, the inverse regression model is introduced, as well as the estimation and prediction procedure. Asymptotic and non asymptotic distribution of parameters estimates are derived in Section~\ref{theory}. Then, confidence region of slope coefficients and prediction regions are established in Section \ref{confidencePredictionRegion}. The finite-sample performance of the proposed confidence and prediction regions are investigated in Section \ref{simulations}, which also includes a comparison with existing methods namely least squares and Lasso. The paper concludes by a discussion in Section \ref{conclusion}.

\section{Inverse regression model}
\label{model}
In this section, we introduce the various elements of the modeling framework.
\subsection{Inverse regression method}


We propose to address the following linear regression problem with random regressors; known as generative model:
\begin{align}
\Xv_i &\sim \mathcal{N}_D(\mathbf{0},\Gammab^\star) \label{eq:model::X}\\
\Yv_i \vert \Xv_i &= \Ab^\star \Xv_i + \boldsymbol{\varepsilon}_i \label{eq:model::Y|X}
\end{align}
where $\Yv=(\Yv_1,\ldots,\Yv_N) \in \mathbb R^{L\times N}$ contains $L$ responses for $N$ subjects and $\Xv=(\Xv_1,\ldots,\Xv_N) \in \mathbb R^{D\times N}$ contains $D$ Gaussian centered predictors with covariance matrix $\Gammab^\star$. The error term $\boldsymbol{\varepsilon} =(\boldsymbol{\varepsilon}_1,\ldots,\boldsymbol{\varepsilon}_N) $ is an unobserved $L\times N$ matrix with independent columns normally distributed, $\boldsymbol{\varepsilon}_1, \ldots, \boldsymbol{\varepsilon}_N \sim \mathcal{N}_L(\mathbf{0}, \Sigmab^\star)$.
The $L\times D$ matrix of slope coefficients  is denoted by $\Ab^\star$.
When $D$ is large or/and when the number of observations  $N$ is smaller than $D$, the so-called least square estimate of $\Ab^\star$ is not numerically computable for the \textit{forward regression} defined in Equations \eqref{eq:model::X} and \eqref{eq:model::Y|X}. Indeed, it requires the inversion of the possibly large matrix $\Xv^\top\Xv$ which is not invertible when $D > N$ and computationally intensive for large D when $N > D$.
An interesting and relatively simple approach to handle this high dimensional problem is to consider the \textit{inverse regression} problem:
\begin{align}
\Yv_i &\sim \mathcal{N}_L(\mathbf{0},\Gammab) \label{eq:model::Y}\\
\Xv_i \vert \Yv_i &= \Ab \Yv_i + \mathbf{e}_i \label{eq:model::X|Y}
\end{align}
where $\Ab$ is a $D \times L$ matrix of slope coefficients of the \textit{inverse regression} and $\mathbf{e}=(\mathbf{e}_1,\ldots,\mathbf{e}_N)$ is a $D \times N$ matrix of unobserved centered Gaussian random noise with residual covariance matrix $\Sigmab$. The inverse regression approach consists in inverting the response and the covariates in the model and performing regression of response on covariates. While least squares estimate is not computable in high dimension for forward regression, it turns out that dealing with the inverse regression problem, under some assumptions on the noise $\mathbf{e}$ detailed hereafter, drastically reduces the number of parameters and makes the problem tractable. 

Note that no intercept is considered in models  \eqref{eq:model::Y} and \eqref{eq:model::X|Y}, which leads to  assume that both response and covariates are centered.

Interestingly, forward parameters $(\Gammab^\star,\Ab^\star,\Sigmab^\star)$ are expressed in function of the inverse parameters $(\Gammab,\Ab,\Sigmab)$ through the following mapping: 
\begin{align}
  \Psi : (\Gammab, \Ab, \Sigmab) \mapsto & (\Gammab^\star, \Ab^\star, \Sigmab^\star)  \nonumber\\
  = & (\Sigmab+{\Ab}\Gammab{\Ab}^{\top}, (\Gammab^{-1}+{\Ab}^{\top}\Sigmab^{-1}{\Ab})^{-1} {\Ab}^{\top}\Sigmab^{-1}, \label{eq:bijection}\\ 
  & (\Gammab^{-1}+{\Ab}^{\top}\Sigmab^{-1}{\Ab})^{-1}). \nonumber
  \end{align}

As $\Psi$ is a one-to-one mapping, estimating the forward regression model, Equations~\eqref{eq:model::X}-\eqref{eq:model::Y|X},  or the inverse regression model, Equations~\eqref{eq:model::Y}-\eqref{eq:model::X|Y}, is equivalent. One can also notice that $\Psi$ is an involution. The advantage of the inverse approach appears when assumptions are made on the large residual covariance matrix $\Sigmab$ in the inverse regression problem of Equations~\eqref{eq:model::Y}-\eqref{eq:model::X|Y}. Indeed, assuming that $\Sigmab$ is diagonal drastically reduces the number of parameters to estimate, while keeping a general modelling. For example, if $D=100$ and $L=5$, the number of parameters to estimate goes from $LD+L(L+1)/2+D(D+1)/2=5565$ in the full model to $LD+L(L+1)/2+D=615$ by assuming that $\Sigmab$ is diagonal. 

\subsection{Estimation}

Considering the inverse model defined in Equations ~\eqref{eq:model::Y}-\eqref{eq:model::X|Y}, the least squares estimators are:
\begin{align}
\widehat{\Gammab} &= \frac{1}{N-1} \Yb^\top \Yb \nonumber\\
\widehat{\Ab} &= (\Yb^\top \Yb)^{-1} \Yb\top \Xb \label{AChap}\\
\forall j\in \{1,\ldots, D\}, \vspace{1em} \widehat{\Sigmab}_{j,j} &= \frac{1}{n-1} \sum_{i=1}^n (\Xb_{i,j}-[\widehat{\Ab}\Yb_i]_j)^2. \nonumber
\end{align}
Then, using $\Psi$, we get straightforwardly estimators for the forward regression:
\begin{align}
\widehat{\Gammab}^\star &= \widehat \Sigmab+\widehat{\Ab} \widehat \Gammab \widehat{\Ab}^{\top} \label{gamma*hat}\\
\widehat{\Ab}^\star &=  (\widehat \Gammab^{-1}+\widehat{\Ab}^{\top}\widehat \Sigmab^{-1}\widehat{\Ab})^{-1} \widehat{\Ab}^{\top}\widehat \Sigmab^{-1} \label{AChapStar}\\
\widehat{\Sigmab}^\star &=  (\widehat \Gammab^{-1}+\widehat{\Ab}^{^\top}\widehat \Sigmab^{-1}\widehat{\Ab})^{-1}\label{sigma*hat}. 
\end{align}
The inverse regression trick allows to compute those estimators even when $D>>N$ as it  requires the inversion of the $L \times L$ matrix $\Yb^T \Yb$ and not the inverse of  $\Xb^T \Xb$.
Moreover, inverse regression is not as computationally intensive as the least squares, because only small or diagonal matrices are inverted: $\widehat{\Gammab}$ is of size $L$ and $\widehat{\Sigmab}$ is diagonal.

\subsection{Prediction of the response}
\label{section::prediction}
Considering those estimators $(\widehat\Ab^\star,\widehat \Gammab^\star,\widehat \Sigmab^\star)$, a new response $\widehat \Yv_{N+1}$ is predicted  for a new observed profile $\xv_{N+1}$ from Model~\eqref{eq:model::Y|X} and defined by:
$$
\widehat \Yv_{N+1} = \widehat \Ab^\star\xv_{N+1}.
$$
In this article, we are interested in studying the uncertainty around this prediction which can be quantified by deriving prediction region. Moreover, we establish the exact distribution of $\widehat \Gammab^\star$ and $\widehat \Sigmab^\star$ and the asymptotic normality of $\widehat\Ab^\star$ which is used to deduce prediction regions.

\section{Theoretical study of the estimators}
            \label{theory}
In this section, we assume that covariance matrices $\Sigmab$ and $\Gammab$ are known. 
Moreover, $\Sigmab$ is supposed to be diagonal, which implies a diagonal + low rank decomposition for $\Gammab^\star$. It allows correlations among covariates. 

%
%
%
Under those assumptions, exact and asymptotic distribution of estimators are derived in this section for the forward regression. 

\subsection{Matrix  normal distribution and Kronecker product}
First we recall some properties about the matrix normal distribution and the tensor product. These results can be found in \cite{GNbook} chapter 2, but every important property is recalled in this paper as we use it extensively.


\begin{defi}[Kronecker product]
Let $A \in M_{m,n}(\mathbb{R})$ and $B \in M_{p,q} (\mathbb{R})$.
Then, the Kronecker product $A \otimes B$ is the $mp \times nq$ block matrix:
\begin{align*}
A \otimes B = \begin{pmatrix} a_{11} B & \ldots, & a_{1n} B \\
\vdots & \ddots & \vdots \\
a_{m1} B& \ldots & A_{mn} B 
\end{pmatrix}.
\end{align*}
\end{defi}

The vectorization is used to work with vectors instead of matrices.
\begin{defi}[Vectorization] The vectorization $\text{vec}(A)$ of a matrix $A$  is a linear transformation which converts the matrix into a column vector, by stacking the columns of the matrix on top of one another.
\end{defi}

As we are interested in the distribution of matrix parameters, the matrix normal distribution is introduced.
\begin{defi}[Matrix  normal distribution]
The random variable
$X \in \mathbb{R}^{L\times D}$ is distributed according to a \emph{matrix  normal distribution} with mean $X_0$ and variances $U\in M_{L,L}(\mathbb{R}) $  (among-row) and $V\in M_{D,D}(\mathbb{R}) $ (among-column), denoted 
$$
X \sim \mathcal{MN}_{LD} (X_0, U, V),
$$
if and only if
$\text{vec}(X) \sim \mathcal{N}_{LD} (\text{vec}(X_0), V\otimes U)$.
\end{defi}

For this distribution, some interesting properties are derived.
\begin{prop}
The following equivalence holds:
\begin{align*}
X \sim \mathcal{MN}_{LD} (X_0, U, V) \hspace{1cm}  \Leftrightarrow \hspace{1cm} X^T \sim \mathcal{MN}_{DL} (X_0^T, V, U).
\end{align*}
\end{prop}

\begin{prop}
If $X \sim \mathcal{MN}_{LD} (X_0, U, V)$, the following properties hold for $A \in \mathcal{M}_{r,D} (\mathbb{R})$ and $B\in \mathcal{M}_{L,s} (\mathbb{R})$
\begin{align*}
AXB &\sim \mathcal{MN}_{rs} (A X_0 B, A U A^T, B^T V B)\\
\text{vec}(AXB) &= (B^T \otimes A) \text{vec}(X)
\end{align*}

For $A \in \mathcal{M}_{r,D} (\mathbb{R}), B\in \mathcal{M}_{L,s} (\mathbb{R}), C \in \mathcal{M}_{r,D} (\mathbb{R}), D\in \mathcal{M}_{L,s} (\mathbb{R})$, the following holds:
\begin{align*}
Cov(\text{vec}(AXB), \text{vec}(CXD)) &= (B^TVD \otimes AUC^T)\\
Cov(\text{vec}(AXB), \text{vec}(CX^TD)) &= (B^T\otimes A) E(\text{vec}(X) \text{vec}(X)^T) T_{LD}^{-1} (D^T \otimes C) \\
&= (B^T V \otimes A U)  T_{LD}^{-1} (D^T \otimes C) \text{ for }X \text{ centered} 
\end{align*}
where $T_{LD}$ is the commutation matrix, transforming the vectorized form of a matrix of size $L \times D$ into the vectorized form of its transpose.
\end{prop}

\subsection[Distribution of  covariance matrices]{Distribution of matrices $\widehat{\Gammab}^\star$ and $\widehat{\Sigmab}^\star$}
 \label{gammabstar}

 In this section,  distributiond of the predictors empirical covariance matrix $\widehat{\Gammab}^\star$  and the residual covaiance matrix $\widehat{\Sigmab}^\star $ are studied.
 
 As $\Sigmab$ and  $\Gammab$ are supposed to be known,  an estimator of  $\widehat{\Gammab}^\star$ is deduced by pluging-in the estimator of $\Ab$ as followed:
 
$$\widehat{\Gammab}^\star =\Sigmab+\widehat{\Ab}\Gammab\widehat{\Ab}^{\top}.$$

The probability density function of $\widehat{\Gammab}^\star$ is derived in the following theorem.

\begin{theorem}[Distribution of $\widehat{\Gammab}^\star$ 
]
Suppose $((\Xb_1,\Yb_1),\ldots, (\Xb_N,\Yb_N))$ is a sequence of $iid$ random variables  satisfying the model defined in Equations \eqref{eq:model::X}-\eqref{eq:model::Y|X}.
Suppose that $\widehat{\Gammab}^\star$ is decomposed as  $\Sigmab+\widehat{\Ab}\Gammab\widehat{\Ab}^{\top}$ where $\widehat{\Ab}$ is the estimator defined Equation~\eqref{AChap}, then the probability density function of $\widehat{\Gammab}^\star$ is defined as,
for symmetric definite positive matrices structured as the sum of a diagonal and a low rank matrix: 
\begin{align*}
& \pi^\frac{-DL+L^2}{2} \{ 2^{\frac12DL}\Gamma_L\left(\frac12L\right) \} ^{-1} \text{det}(\Sigmab)^{-\frac12L}\text{det}(\Bb)^{-\frac12D}\text{etr}\left( -\frac12 \Sigmab^{-1} \Ab \Yv^T\Yv\Ab^T\right) \nonumber \\
& \text{etr}\left( -\frac12q \Sigmab (\Gammab^\star - \Sigmab) \right) \left( \prod_{\lambda >0} \lambda(\Gammab^\star - \Sigmab)\right)^{\frac12(L-D-1)} \sum_{k=0}^\infty\sum_{\kappa} \frac{1}{(\frac12L)_\kappa k!} \nonumber \\
& P_\kappa\left( \frac{1}{\sqrt{2}}\Sigmab^{-\frac12} \Ab (\Yv^T\Yv)^{\frac12} (\mathbb I_L -q\Bb)^{-\frac12},\Bb^{-1}-q\mathbb I_L,\frac12 \Sigmab^{-\frac12} (\Gammab^\star - \Sigmab) \Sigmab^{-\frac12}\right) \nonumber
\end{align*}
 where $\lambda(A)$ corresponds to the eigenvalues of $A$ and $\Bb = (\Yv^T\Yv)^{-\frac12} \Gammab (\Yv^T\Yv)^{-\frac12} $, and $q > 0$ an arbitrary constant such that $\mathbb I_L -q\Bb$ is positive definite, and $\Gamma_L(\cdot)$, $\text{etr}(\cdot)$ and the Hayakawa polynomial $P_\kappa(\cdot,\cdot,\cdot)$ defined as in Appendix~\ref{sec:notations}.
\label{theoDistribGammaStar}
\end{theorem}

Note that this distribution is related to a Wishart distribution with a rescaling related to $\Gammab$ and a translation of $\Sigmab$.
The proof is available in Appendix \ref{proofDistribGammaStar} and mainly uses the law of the unconscious statistician and matricial computation.

Note that response and covariates play a symmetric role in inverse regression as their role are inverted for estimation. However, interestingly, the following theorem involves a standard Wishart-like distribution while the previous one involves a singular Wishart-like distribution even if they consist in finding the distribution of matrices with similar decomposition.

In the same way, the density distribution of residual empirical covariance matrix $\widehat{\Sigmab}^\star$ is deduced.

\begin{theorem}[Distribution of $\widehat{\Sigmab}^\star$]
\label{theoDistribSigmaStar}
Suppose $((\Xb_1,\Yb_1),\ldots, (\Xb_N,\Yb_N))$ is a sequence of $iid$ random variables  satisfying the model defined in Equations \eqref{eq:model::X} and \eqref{eq:model::Y|X}.
Suppose that $\widehat{\Sigmab}^\star$ is decomposed as  $( \Gammab^{-1}+\widehat{\Ab}^{\top} \Sigmab^{-1}\widehat{\Ab})^{-1}$ where $\widehat{\Ab}$ is the estimator defined Equation~\eqref{AChap}, then the probability density function of $\widehat{\Sigmab}^\star$ is defined as:
 \begin{align*}
 & \{ 2^{\frac12DL}\Gamma_L\left(\frac12D\right) \} ^{-1} \text{det}((\Yv^T\Yv)^{-1})^{-\frac12D}\text{det}(\Sigmab^\star)^{(L+1)} \nonumber \\
 &\text{etr}\left( -\frac12 (\Yv^T\Yv) \Ab^T \Sigmab^{-1}\Ab\right) \nonumber \text{etr}\left( -\frac12q \Yv^T\Yv \left((\Sigmab^\star)^{-1} - \Gammab^{-1}\right)\right)  \\
 &\text{det}((\Sigmab^\star)^{-1} - \Gammab^{-1})^{\frac12(D-L-1)} \sum_{k=0}^\infty\sum_{\kappa} \frac{1}{(\frac12D)_\kappa k!} \nonumber \\
& P_\kappa\left( \frac{(1-q)^{-\frac12}}{\sqrt{2}}(\Yv^T\Yv)^\frac12 \Ab \Sigmab^{-\frac12},(1-q)\mathbb I_D,\frac12 (\Yv^T\Yv)^\frac12 ((\Sigmab^\star)^{-1} - \Gammab^{-1}) (\Yv^T\Yv)^\frac12\right) \nonumber
 \end{align*}
 where $\Gamma_L(\cdot)$ is the multivariate gamma function, $\text{etr}(\cdot)$ is the exponential of the trace of a matrix and $P_\kappa(\cdot,\cdot,\cdot)$ is the generalized Hayakawa polynomial. These notations are more precisely defined in Appendix~\ref{sec:notations}. 
\end{theorem}
Proof is available in Appendix \ref{proofDistribSigmaStar} with a similar approach of Theorem~\ref{theoDistribGammaStar}. Note that confidence interval for covariance matrices $\Gammab^\star$ and $\Sigmab^\star$ can be derived as the exact distribution of their estimators are known using the previous theorems. Moreover, the exact distribution of $\widehat{\Ab}$ and $\widehat{\Sigmab^\star}$ are known making the exact distribution of $\widehat{\Ab^\star}$ accessible. However, computing this distribution is strong analytically and algorithmically, so in the following section, we focus on the asymptotic normality of $\widehat{\Ab^\star}$.

\subsection[Asymptotic normality of regression coefficients]{Asymptotic normality of $\widehat{\Ab}^\star$}

In order to derive the asymptotic normality of the forward regression coefficients $\widehat{\Ab}^\star$, the distribution of the inverse regression coefficients matrix $\widehat{\Ab}$ is described at first.
\begin{prop}[Distribution of $\widehat{\Ab}$]
Suppose $((\Xb_1,\Yb_1),\ldots, (\Xb_N,\Yb_N))$ is a sequence of $iid$ random variables  satisfying the model defined in Equations \eqref{eq:model::X} and \eqref{eq:model::Y|X}, or equivalently in Equations \eqref{eq:model::Y} and \eqref{eq:model::X|Y}.
Then, 
\label{distribAchap}
$$\widehat{\Ab} \sim \mathcal{MN}_{DL} (\Ab, \Sigmab, (\Yv^T \Yv)^{-1}). $$ 
\end{prop}

This result is an extension of the least square estimator in the multivariate linear model to the multiple multivariate linear model. The proof is straightforward.

From this, we derive the asymptotic normality of $\widehat{\Ab}^\star$.
A matricial version of the $\Delta$-method is used,
which involves the differential of the function $g:\Ab \mapsto \Ab^\star$ and the corresponding asymptotic variance of $\widehat\Ab^\star$. They are first computed in the following lemma.

\begin{lemma}
\label{computeDiff}
Suppose $((\Xb_1,\Yb_1),\ldots, (\Xb_N,\Yb_N))$ is a sequence of $iid$ random variables  satisfying the model defined in Equations \eqref{eq:model::X} and \eqref{eq:model::Y|X}.
Let 
\begin{align}
g : \mathbb{R}^{D \times L} &\rightarrow \mathbb{R}^{L \times D} \nonumber\\
 \Ab &\mapsto \Ab^\star =  \Sigmab^\star\Ab^T\Sigmab^{-1} = (\Gammab^{-1} + \Ab^T\Sigmab^{-1}\Ab)^{-1}\Ab^T\Sigmab^{-1}\label{diff}
\end{align}
Then the differential of this function at point $(\widehat \Ab - \Ab)$ is,
\begin{align}
    Dg&(\Ab) . (\widehat \Ab - \Ab) = \label{eq:diffg} \\
    & \Sigmab^\star (\widehat \Ab - \Ab)^T\Sigmab^{-1} - \Sigmab^\star (\widehat \Ab - \Ab)^T\Sigmab^{-1} \Ab \Ab^\star - \Ab^\star (\widehat \Ab - \Ab) \Ab^\star. \nonumber
\end{align}
Moreover, the covariance of this random matrix is given by the following:
\begin{align}
Cov&(\text{vec}(Dg(\Ab) . (\widehat \Ab - \Ab))) 
= \label{covVecDiff} \\
& \left((\Sigmab^{-1} + (\Ab^\star)^T \Ab^T \Sigmab^{-1} \Ab \Ab^\star -2 \Sigmab^{-1} \Ab \Ab^\star)\otimes \Sigmab^\star \Gammab \Sigmab^\star\right) \nonumber\\
&+ ((\Ab^\star)^T \Gammab \Ab^\star \otimes \Ab^\star \Sigmab (\Ab^\star)^T ) \nonumber\\
&-2 \left(
(\mathbf{I} \otimes \Sigmab^\star \Gammab)+
((\Ab^\star)^T \Ab^T \otimes \Sigmab^\star \Gammab) \right)  T^{-1}_{LD}ÃÂ  ((\Ab^\star)^T \otimes \Ab^\star).\nonumber
\end{align}

\end{lemma}
Proof of Lemma  \ref{computeDiff} is given in Appendix \ref{proofLemma1}.

Finally,  the following theorem, which is the key of this paper, details the distribution of $\widehat{\Ab}^\star$.
\begin{theorem}[Asymptotic normality of $\widehat{\Ab}^\star$]
Suppose $((\Xb_1,\Yb_1),\ldots, (\Xb_N,\Yb_N))$ is a sequence of $iid$ random variables  satisfying the model defined in Equations \eqref{eq:model::X} and \eqref{eq:model::Y|X}.
Let 
\begin{align}
g : \mathbb{R}^{D \times L} &\rightarrow \mathbb{R}^{L \times D} \nonumber\\
 \Ab &\mapsto \Ab^\star =  \Sigmab^\star\Ab^T\Sigmab^{-1} = (\Gammab^{-1} + \Ab^T\Sigmab^{-1}\Ab)^{-1}\Ab^T\Sigmab^{-1}\label{g}
\end{align}
Then, the following holds for the estimator $\widehat{\Ab}^\star$ defined in Equation \eqref{AChapStar}.
\begin{align*}
 \sqrt{N}(\text{vec}(\widehat{\Ab}^\star) - \text{vec}(\Ab^\star))& \underset{{N\rightarrow +\infty} }{\rightarrow}\mathcal{N}_{DL}(\mathbf{0}, \Theta(\Ab))
\end{align*}
where $\Theta(\Ab) = Cov(\text{vec}(Dg(\Ab) . (\widehat \Ab - \Ab)))$ defined in Equation \eqref{covVecDiff}.

Moreover, $\Theta(\widehat{\Ab})$ is a consistent estimator of $\Theta(\Ab)$, then by Slutsky's Lemma we get the following:
\begin{align}
\sqrt{N}  (\text{vec}(\widehat \Ab^\star) - \text{vec}( \Ab^\star))^T \Theta(\widehat{\Ab})^{-1}  (\text{vec}(\widehat \Ab^\star) - \text{vec}( \Ab^\star)) \underset{N \rightarrow +\infty}{\rightarrow} \chi^2_{DL}.\label{asympNormalityAChapStarSlustky}
\end{align}
\end{theorem}

\begin{proof}
The matrix version of the $\Delta$-method is  a second order Taylor expansion of $g: \Ab \mapsto \Ab^\star$.
Therefore, for $\Ab \in M_{D,L}(\mathbb{R})$ and $g$ defined by Equation~\eqref{g}, the Taylor expansion leads to 
\begin{align*}
\widehat{\Ab}^\star = g(\widehat \Ab) = g(\Ab) + Dg(\Ab) . (\widehat \Ab - \Ab) + R_N(\widehat \Ab)
\end{align*}
with $R_N(\widehat \Ab)$ is a rest term and $Dg(\Ab) . (\widehat \Ab - \Ab)$ is given in Lemma~\ref{computeDiff}.


Then, 
\begin{align}
\sqrt{N} (\widehat{\Ab}^\star - {\Ab}^\star) &=  \sqrt{N} Dg(\Ab).(\widehat{\Ab}-\Ab) + \sqrt{N} R_N(\widehat{\Ab}) \label{TaylorExpansion}
\end{align}
The last term in \eqref{TaylorExpansion} converges to 0 in probability, and by Proposition \ref{distribAchap}, the linear combination with respect to $\hat{\Ab}$
defined in \eqref{eq:diffg} is a multivariate Gaussian, centered.
Using \eqref{covVecDiff}, we get the distribution of the vectorized vector $vec(\widehat{\Ab}^\star)$.

Limiting distribution \eqref{asympNormalityAChapStarSlustky} is get by using Slustky's Lemma, as $\widehat{\Ab}$  converges in probability to $\Ab$.
\end{proof}

This results is the key theorem of this article as it allows to derive confidence regions for $\Ab^\star$ and prediction regions.
Wheres we consider the vectorize matrix $\widehat{\Ab}^\star$, formulae are explicit. Remark that the degree of freedom of the $\chi^2$ distribution depends on the size of the response and the covariates in the same way.

\section{Confidence regions and predictions regions}
\label{confidencePredictionRegion}
In this section, we provide confidence regions for $vec(\Ab^\star)$ and prediction regions for $\yb$ through the inverse regression method.

\subsection[Confidence regions for regression coefficients]{Confidence regions for $\Ab^\star$}
\begin{theorem}
 Suppose $((\Xb_1,\Yb_1),\ldots, (\Xb_N,\Yb_N))$ is a sequence of $iid$ random variables  satisfying the model defined in Equations \eqref{eq:model::X} and \eqref{eq:model::Y|X}.
 Then, a confidence region for $\Ab^\star$ is
 $$
 P\left( \text{vec}(\Ab^\star ) \in \tilde{\mathcal{R}}_{\text{vec}(\Ab^\star), \alpha} \right) \underset{n \rightarrow +\infty}{\rightarrow} 1-\alpha$$
where 
\begin{align*}
    \tilde{\mathcal{R}}_{\text{vec}(\Ab^\star), \alpha} = &\left\{ \ab^\star \in M_{L,D}(\mathbb{R}) \text{ s.t. } \right. \\ 
    & \left. (\text{vec}(\ab^\star-\widehat\Ab^\star))^T \Theta({\Ab})^{-1} (\text{vec}(\ab^\star-\widehat\Ab^\star) )
 \leq \chi^2_{DL}(1-\alpha) \right\}.
\end{align*}

with $\Theta(\Ab) = Cov(\text{vec}(Dg(\Ab) . (\widehat \Ab - \Ab)))$ defined in Equation \eqref{covVecDiff}.
\end{theorem}
Note that this confidence region is a quadratic form as matrix parameters are considered. Then, he $\chi^2$ distribtion is involved.
Those explicit formulae allows to compute confidence regions in practice. Numeric performances stand in Section \ref{simulations}.



\subsection{Prediction regions}
\begin{theorem}
 Suppose $((\Xb_1,\Yb_1),\ldots, (\Xb_N,\Yb_N))$ is a sequence of $iid$ random variables  satisfying the model defined in Equations \eqref{eq:model::X} and \eqref{eq:model::Y|X}.
 Then,
$$
 P\left( \yb_{n+1} \in \widetilde{\mathcal{PR}}_{\yb, \alpha} \right) \underset{n \rightarrow +\infty}{\rightarrow} 1-\alpha$$
where 
\begin{align}
\widetilde{\mathcal{PR}}_{\yb, \alpha}& = \left\{ y\in \mathbb R^L \text{ s.t. }  \right.\label{eq:ICAstar} \\
& \left. (y-\widehat\Ab^\star \xb_{N+1})^T (\Omega({\Ab}^\star \xb_{N+1}) + \Sigmab^\star)^{-1} (y-\widehat\Ab^\star\xb_{N+1}) 
 \leq \chi^2_{L}(1-\alpha) \right\} \nonumber
\end{align}
 where 
 $\Omega(\widehat{\Ab}^\star \xb_{N+1})$ 
 is the following $(L \times L)$ covariance matrix
 $$
\Omega({\Ab}^\star \xb_{N+1}) =(\mathbb I_L \otimes \xb_{N+1} ^T  ) \Theta({\Ab}) (\xb_{N+1} ^T \otimes \mathbb I_L).
$$

where $\Theta(\Ab) = Cov(\text{vec}(Dg(\Ab) . (\widehat \Ab - \Ab)))$ defined in Equation \eqref{covVecDiff}.
\end{theorem}

One can notice that the covariance matrix that is inverted in Equation~\eqref{eq:ICAstar} breaks down into 2 parts. The first one, $\Omega({\Ab}^\star \xb_{N+1})$, represents the variance of the prediction which depends on the estimation accuracy of $\Ab^\star$ while the second part, $\Sigmab^\star$, is the variance inherited from the residuals.

Moreover, as previously, every formula is explicit so numerical experiments are derived in Section \ref{simulations}.

\section{Simulations}
\label{simulations}

The goal of this section is to compute the prediction regions derived from the theoretical results presented in Section \ref{confidencePredictionRegion}. For several designs regarding the sample size, the dimension, the sparsity and several covariance patterns, we study the coverage, the volume of the interval and the computation time.  For comparison, we also compute prediction intervals deduced from the least square estimator and a regularized approach. A \texttt{R} code is available on authors' webpages to apply the 3 compared methods on simulated data, on the following webpage \href{https://research.pasteur.fr/fr/member/emeline-perthame/}{https://research.pasteur.fr/fr/member/emeline-perthame/}. 


\subsection{Simulation design}
\label{designSimu}

In order to assess the impact of data dimension and design complexity on different estimation methods of prediction regions, we perform a simulation study. We consider a response with dimension $L$ varying in $\{1,2,5\}$. Indeed, when $L=1 \text{ or } 2$, prediction regions are easily graphically displayable which is useful to visualize methods. We focus on three distinct designs namely a high-dimensional one $(N=50, D=100)$, an asymptotic one $(N=500, D=100)$ and an intermediate design $(N=100, D=100)$ which allows to investigate situations with $N \leq D$ and $N > D$. Data are simulated according to an inverse regression model and forward parameters are deduced from Equation \eqref{eq:bijection}. For each combination of dimension, we focus on the 3 following scenarii: 
\begin{itemize}
\item[(Case 1)] Sparse regression coefficients and independent responses: $\Ab$ is a $D \times L $ matrix with $90\%$ of zero entries randomly drawn. The $10\%$ nonzero remaining coefficients are uniformly drawn into a uniform distribution on $(-2,2)$. Matrix $\Gammab$ of covariances between response terms is set to $\mathbb I_L$. The residual covariance matrix of inverse regression $\Sigmab$ is set to $\mathbb I_D$. Note that a diagonal $\Sigmab$ and a sparse $\Ab$ under the inverse model lead to a sparse matrix of regression coefficients for forward regression $\Ab^\star$. 
\item[(Case 2)] Sparse regression coefficients and correlated responses: same as previous scenario except that $\Gammab$ is a full covariance matrix generated according to a factor model such as dependence among response terms is rather strong. 
\item[(Case 3)] Full matrix of regression coefficients and correlated responses: coefficient matrix $\Ab$ is full with entries uniformly sampled in $[-0.5,0.5]$ and covariance matrix $\Gammab$ is generated as in Case 2. The residual covariance matrix $\Sigmab$ is set to $\mathbb I_D$
\end{itemize}
Note that the amplitude of coefficients in $\Ab$ differs from one case to another. This amplitude is adjusted in order to make scenarii comparable regarding to the signal to noise ratio (SNR) criterion defined as: 
$$
\text{SNR} = \frac{1}{L} \text{trace}(\Ab^\star \Gammab^\star (\Ab^\star){^T}(\Sigmab^\star)^{-1})
$$
where trace refers to the sum of diagonal entries of a matrix. In this simulation setting, for all cases and all values of $L$, the SNR varies between 5 and 10 which is rather (reasonably) high. Note that we extended the well-known SNR definition of \cite{VG17} to our multivariate response framework.

Datasets are generated under a linear regression model as defined in Equations \eqref{eq:model::X}-\eqref{eq:model::Y|X}. For each simulated design, $1~000$ learning datasets with dimension $(N,D)$ are generated as well as $1~000$ corresponding testing observations. 
Note that the computation of prediction regions for inverse model involves the computation of a commutation matrix. To compute such matrices, we used the fast routine implemented in the function \texttt{commutation.matrix} available in the \texttt{R} package \texttt{matrixcalc}.

We compare the prediction regions derived from the 3 following methods: the proposed method based on inverse regression refered as IR in the following, the so-called least square estimator (LSE) for designs with $N > D$ and a lasso prediction interval based on bootstrap for designs with $N < D$. The accuracy of the method is assessed by computing the coverage (proportion of testing observations falling into the prediction region), the volume of the prediction regions and the computation time required to compute the prediction region on a MacBook Pro - 2,9 GHz Intel Core i5 processor - RAM 16 Go with programs written in \texttt{R}. In this simulation study, the level of confidence for prediction regions is set to 95\%. 

\bigskip



\subsection{Results of the intensive simulation study}

The results of this simulation study are presented in Table~\ref{tab:addlabel}. This table presents the results for varying sample sizes and designs in column, and coverage, volume and time computation in row for varying methods and response dimension. 
For each scenario, IR is compared to LSE when $N > D$ and to Lasso when $N \leq D$.

First, Table~\ref{tab:addlabel} demonstrates that IR performs as well as a variable selection method. Indeed, its performances are similar or even better than Lasso for multivariate response: IR achieves larger coverage and smaller volume.  Note that multivariate version of the Lasso is not implemented to our knowledge in \texttt{R} which makes IR  a challenging method. Interestingly, IR, which does not suppose sparsity in the model, seems to be efficient on sparse design (Cases 1 and 2)  regarding to both coverage and volume. Table~\ref{tab:addlabel} also illustrates that our results are asymptotic, meaning that performances of IR are good regarding volume and coverage for $N > D$. When $N < D$, the confidence level increases with $N$ and is reached when $N > D$. Note that the confidence level is almost reached for $N=D$ which suggests that the asymptotic normality may be quickly reached. Compared to bootstrapped Lasso, IR approach is significantly faster as our method does not rely on resampling. At last, this table shows that IR works well in high-dimension as large $D$ and $N$ are computationally feasible. Computation time is reasonable while achieving challenging coverage and volume when both $D$ and $N$ are large.

Whatever the design, note that the volume of prediction regions increases with $L$, meaning the underlying space dimension. It is interesting to notice that, by normalising the volume by the dimension, the volume stays constant across the situations studied. 



\begin{table}
\begin{center}
\scalebox{0.8}{
    \begin{tabular}{cllccccccccc}
\hline
          &       &       &       & N = 50 &       &       & N = 100 &       &       & N = 500 &  \\
          &       &       & Case 1 & Case 2 & Case 3 & Case 1 & Case 2 & Case 3 & Case 1 & Case 2 & Case 3 \\
\hline
          & IR & \multicolumn{1}{l}{Coverage} & 0.88 & 0.87 & 0.84 & 0.92 & 0.93 & 0.93 & 0.95 & 0.94 & 0.95 \\
          & Lasso/LSE &       & 0.86  & 0.88  & 0.86  & 0.90  & 0.92  & 0.88  & 0.94  & 0.95  & 0.95 \\
    \multicolumn{1}{l}{L = 1} & IR & \multicolumn{1}{l}{Volume} & 1.26 & 1.26 & 1.25 & 1.28 & 1.28 & 1.27 & 1.29 & 1.29 & 1.29 \\
          & Lasso/LSE &       & 1.31  & 1.28  & 1.53  & 1.26  & 1.25  & 1.35  & 1.45  & 1.45  & 1.44 \\
          & IR & \multicolumn{1}{l}{CPU} & 0.02 & 0.02 & 0.02 & 0.02 & 0.02 & 0.02 & 0.02 & 0.02 & 0.02 \\
          & Lasso/LSE &       & 1.01  & 0.97  & 1.09  & 1.13  & 1.13  & 1.35  & 0.01  & 0.01  & 0.01 \\
          & IR & \multicolumn{1}{l}{Coverage} & 0.86 & 0.84 & 0.86 & 0.91 & 0.91 & 0.90 & 0.94 & 0.94 & 0.95 \\
          & Lasso/LSE &       & 0.86  & 0.86  & 0.77  & 0.89  & 0.90  & 0.81  & 0.95  & 0.95  & 0.95 \\
    \multicolumn{1}{l}{L = 2} & IR & \multicolumn{1}{l}{Volume} & 1.90 & 1.90 & 1.92 & 1.96 & 1.94 & 1.99 & 2.00 & 1.96 & 2.03 \\
          & Lasso/LSE &       & 2.07  & 2.05  & 3.00  & 1.95  & 1.93  & 2.31  & 2.51  & 2.50  & 2.55 \\
          & IR & \multicolumn{1}{l}{CPU} & 0.09 & 0.09 & 0.10 & 0.09 & 0.09 & 0.09 & 0.09 & 0.15 & 0.09 \\
          & Lasso/LSE &       & 2.06  & 2.11  & 2.27  & 2.53  & 2.61  & 2.82  & 0.01  & 0.01  & 0.01 \\
          & IR & \multicolumn{1}{l}{Coverage} & 0.84  & 0.81  & 0.84  & 0.92  & 0.91  & 0.90  & 0.94  & 0.94  & 0.95 \\
          & Lasso/LSE &       & 0.75 & 0.77 & 0.74 & 0.87 & 0.89 & 0.88 & 0.94 & 0.94 & 0.95 \\
    \multicolumn{1}{l}{L = 5} & IR & \multicolumn{1}{l}{Volume} & 6.78 & 6.34 & 7.52 & 7.24 & 6.74 & 8.13 & 7.27 & 6.84 & 8.26 \\
          & Lasso/LSE &       & 9.22  & 8.91  & 24.89 & 7.36  & 6.89  & 11.93 & 12.60 & 12.09 & 14.59 \\
          & IR & \multicolumn{1}{l}{CPU} & 1.76 & 1.16 & 1.37 & 1.69 & 1.41 & 1.25 & 1.27 & 1.24 & 1.29 \\
          & Lasso/LSE &       & 5.16  & 4.73  & 5.35  & 5.48  & 5.50  & 6.62  & 0.01  & 0.01  & 0.01 \\
          \hline
    \end{tabular}%
    }
\caption{Results of simulations study: 
prediction regions computed on datasets simulated under models described in Section \ref{designSimu}.
Coverage, volume and CPU time are computed for each method to compare performances.
For large sample size, we compare IR with the LSE and for small size we compare IR with the bootstrapped Lasso.
Each method is assessed 1000 times, and mean is computed.}
  \label{tab:addlabel}%
  \end{center}
\end{table}%

Figure~\ref{fig:ellipse} displays a graphical representation of prediction regions for Case 1 which are ellipses when $L=2$. We consider two sample sizes, $N=50$ and $N=500$. Dotted line represents ellipses computed by LSE when $N=500$ and Lasso when $N=50$, long dashed line represents ellipses computed by IR and solid line represents true prediction regions computed with true parameters used for simulation. Grey dots are 500 replications of responses from the same covariate's profile representing the residual variance. Three specific profiles of covariates are considered: on the left panel, prediction ellipse for the median covariate's profile is computed which is an easy situation. When $N=500$, both LSE and IR provide similar ellipses, close to the true one. When $N=50$, IR's ellipse is close to the true one while lasso correctly predicts the response but the volume of the ellipse is larger. For the middle panel, a covariate's profile corresponding to quantile 0.35 is generated making the computation of the prediction ellipse more complex. When sample size is large, LSE and IR are competitive regarding to true ellipse and equivalent. When $N=50$, the ellipse computed with IR is larger than the theoretical one. The bootstrapped Lasso fails in prediction, which confirms the lower coverages observed in Table~\ref{tab:addlabel}. At last, for the right panel, an even more extreme profile associated to quantile 0.2 is generated, making the computation less reliable. When $N=500$, the volume of ellipses computed by LSE and IR gets even larger as the covariate's profile gets far from the mean. Notice that LSE and IR again achieve similar ellipses in this setting. When $N=50$, conclusions of the middle panel apply as well.

\begin{figure}
\includegraphics[width=\textwidth]{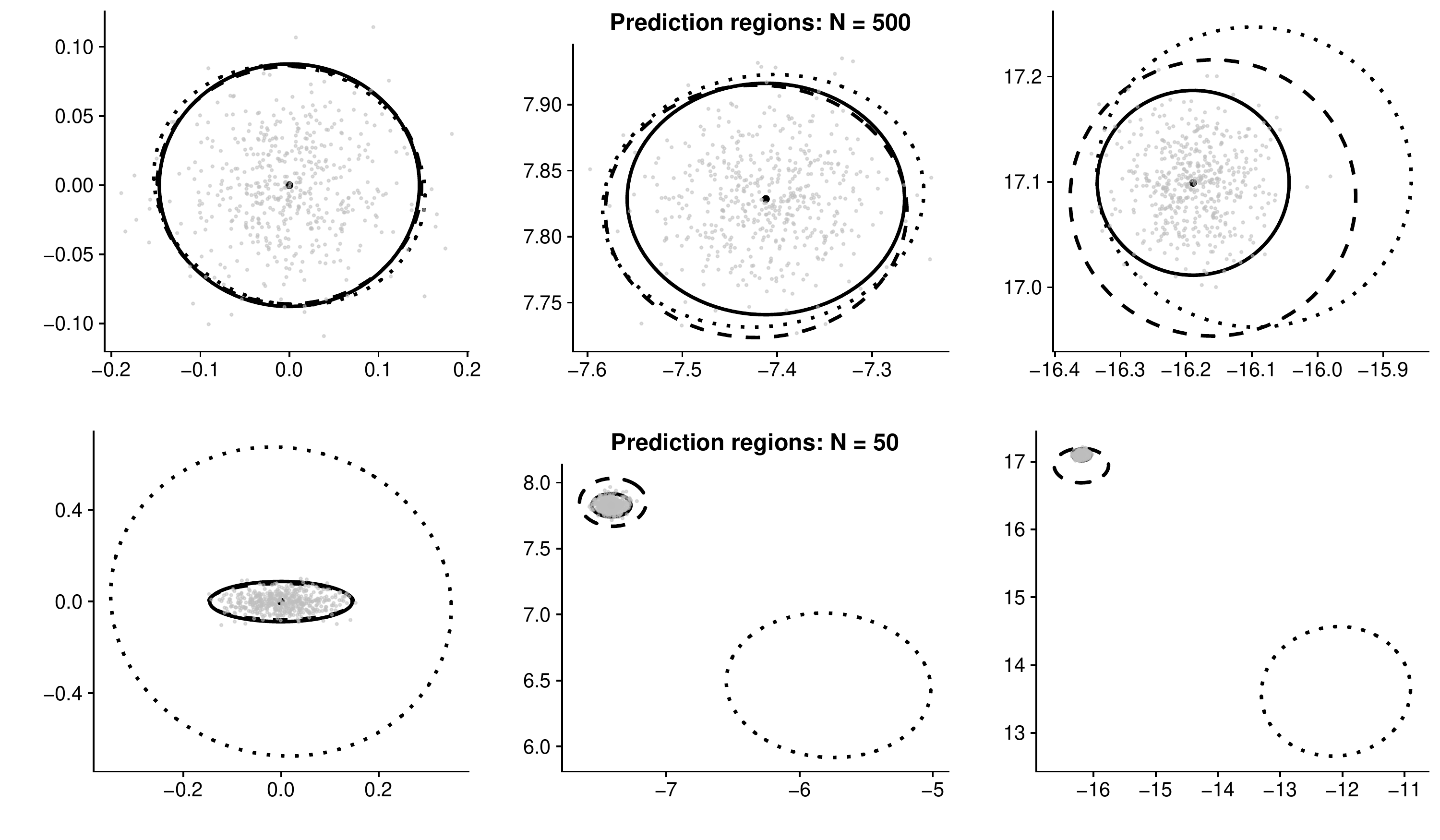}
\caption{Prediction regions for $L=2$. Dotted line: LSE for $N=500$ and Bootstrapped Lasso for $N=50$, long dashed line: IR, solid line: true parameters, grey dots: 500 responses generated from the same covariate's profile.   \label{fig:ellipse}}
\end{figure}

\subsection{Study of estimation accuracy}

In this section, we focus on the first setting (Case 1) with $L=2$ and $D=5$ and $N=100$ in order to visualise the ability of inverse regression to estimate parameters $(\Ab^\star, \Gammab^\star,\Sigmab^\star)$ and to predict response. Violin plots of Figures~\ref{fig:Gammastar} to \ref{fig:Astar} display the distribution of the estimators in black and the true value of the parameter in red. Regarding the estimation of the $D \times D$ matrix $\Gammab^\star$, Figure~\ref{fig:Gammastar} demonstrates that IR is able to retrieve the diagonal structure of the true matrix. Note that the estimation is more variable for diagonal terms. Same remarks hold for the estimation of the $L \times L$ matrix $\Sigmab^\star$, see Figure~\ref{fig:Sigmastar}. Regarding estimation of $\Ab^\star$, it is interesting to notice that IR partially retrieves the sparse structure of the true parameter. Indeed, all values in $\Ab^\star$ are zero except the 4th coefficient of the first row, and the 3rd value of the second row in Figure~\ref{fig:Astar}. The corresponding violin plots are centred around the true value. 

\begin{figure}
    \begin{center}
    \includegraphics[width=0.8\textwidth,trim={0 0 0 1.1cm},clip]{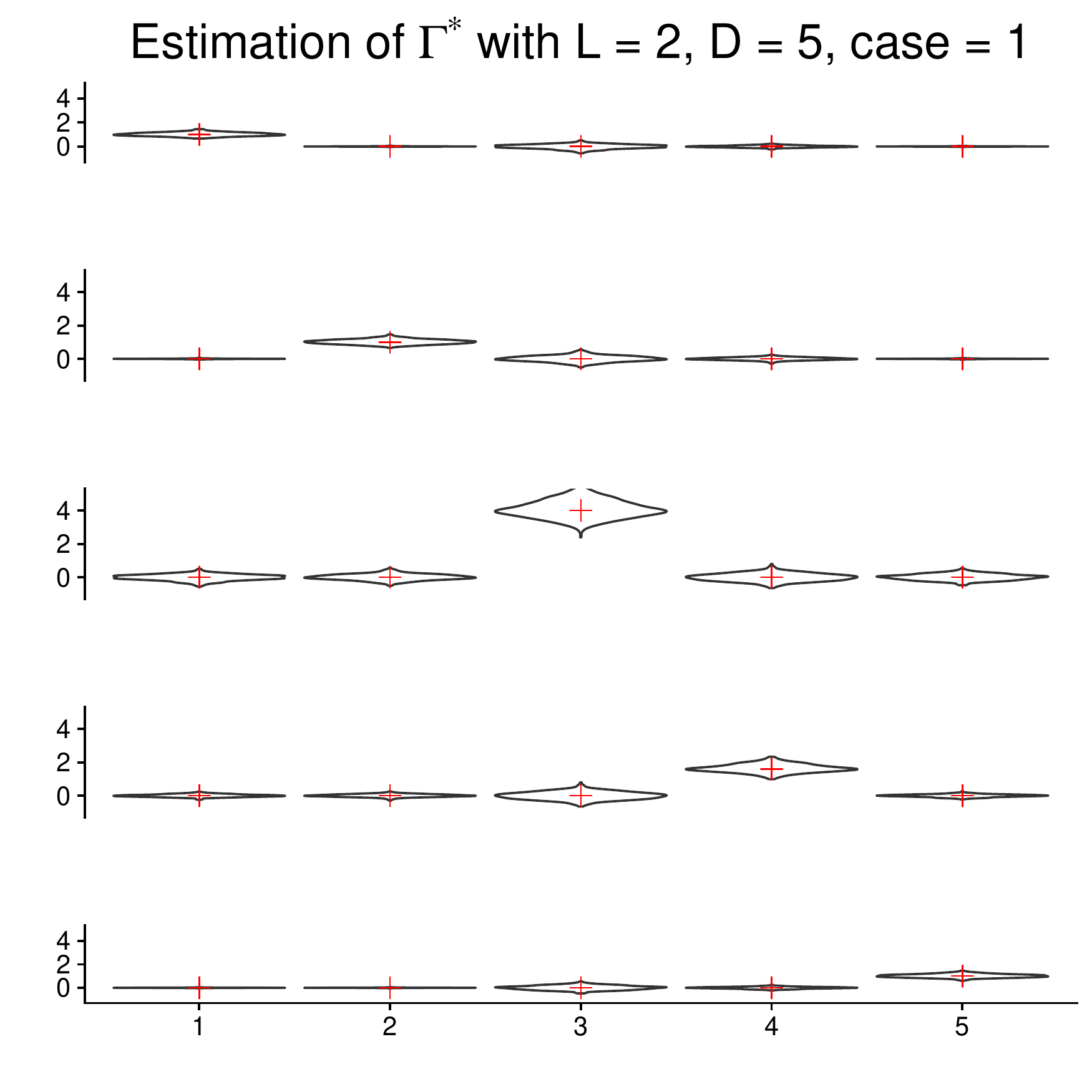}
    \caption{Violin plots displaying the distribution of $\Gammab^\star$ estimator for $L=2, D=5$ and Case 1. $\Gammab^\star$ is diagonal, true values are located by red crosses. \label{fig:Gammastar}}
    \end{center}    
\end{figure}

\begin{figure}
    \begin{center}
    \includegraphics[width=0.5\textwidth,trim={0 0 0 1.1cm},clip]{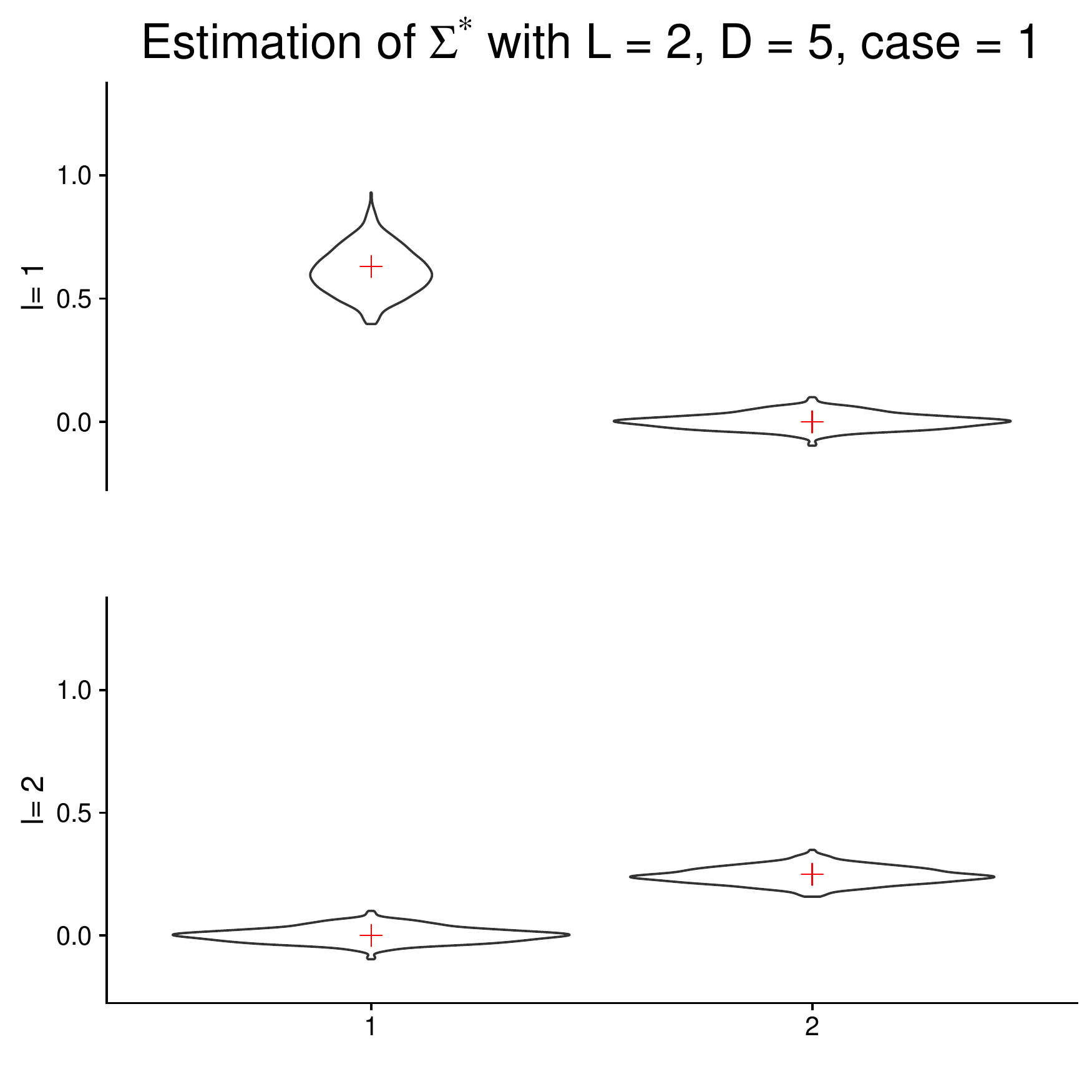}
    \caption{Violin plots displaying the distribution of $\Sigmab^\star$ estimator for $L=2, D=5$ and Case 1. $\Sigmab^\star$ is diagonal, true values are located by red crosses.\label{fig:Sigmastar}}
    \end{center}
\end{figure}

\begin{figure}
    \begin{center}
    \includegraphics[width=\textwidth,trim={0 0 0 1.1cm},clip]{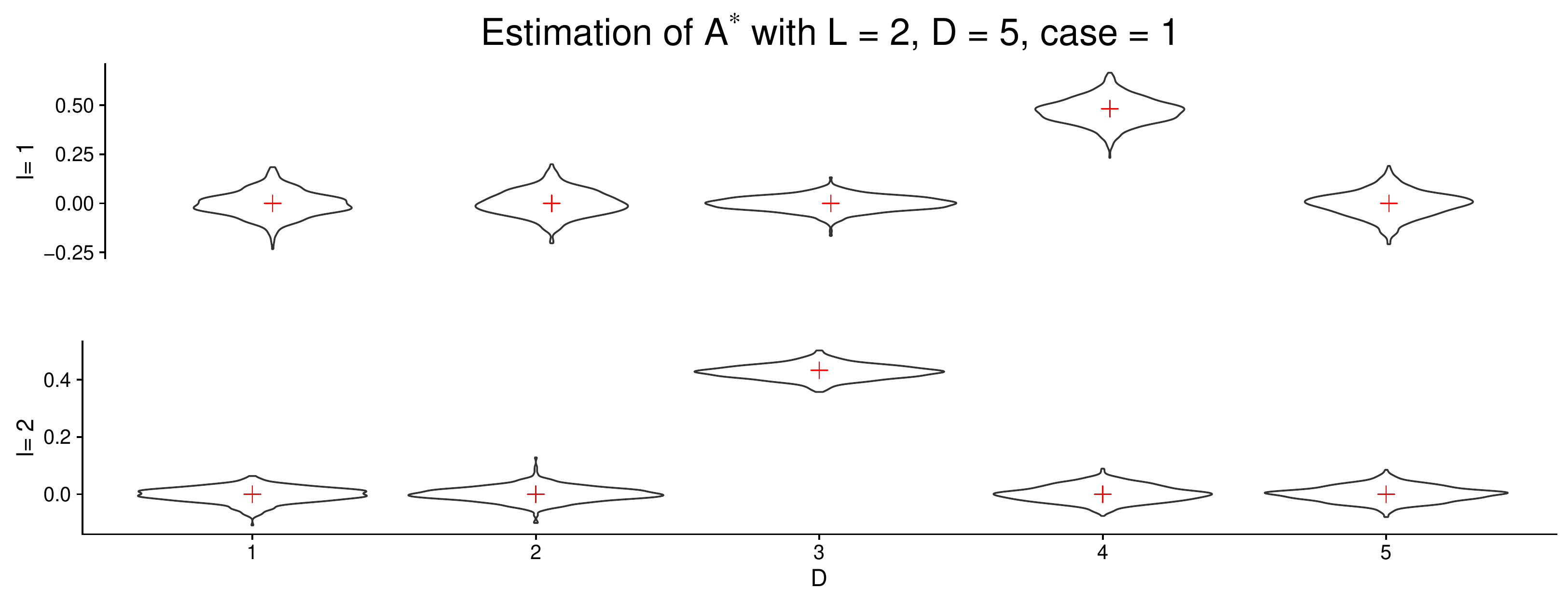}
    \caption{Violin plots displaying the distribution of $\Ab^\star$ estimator for $L=2, D=5$ and Case 1. $\Ab^\star$ is sparse, with 2 non zero entries, true values are located by red crosses.\label{fig:Astar}}
    \end{center}
\end{figure}

Figure~\ref{fig:ychap} displays the distribution of absolute prediction error $\vert \widehat \Yv - \Yv\vert$. Note that IR achieves interesting prediction accuracy most of prediction errors are close to 0. Prediction error of the second response seems easier to predict than the first component which is not surprising as the residual variance in matrix $\Sigmab^\star$ for the 2nd response is smaller than residual variance of first response component. 

\begin{figure}
    \begin{center}
    \includegraphics[width=\textwidth,trim={0 0 0 1.3cm},clip]{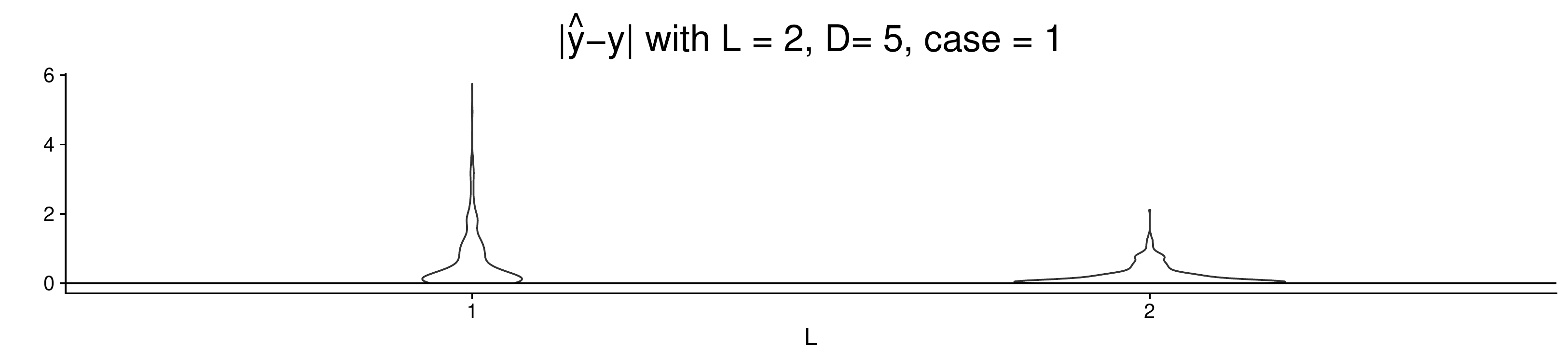}
    \caption{Violin plots displaying the distribution of the absolute prediction error $\vert \widehat \Yv - \Yv \vert $ for $L=2, D=5$ and Case 1. \label{fig:ychap}}
    \end{center}
\end{figure}

\section{Conclusion and further discussion}
\label{conclusion}

In this article, the properties of inverse regression are extensively investigated. Inverse regression addresses linear regression issues with random multivariate predictors and multiple responses. The characteristic of this model is that it inverts the role of covariates and response. By making weak assumptions on the residual covariance matrix of the inverse regression, this model allows to consider settings with both large sample size and covariates dimension, as an alternative to least square methods or regularized methods. Explicit estimators of model parameters are derived, for which exact or asymptotic distributions and confidence regions are deduced. Last but not least, asymptotic prediction regions are derived, allowing to quantify the confidence in estimation.

In an intensive simulation study, we present inverse regression as an alternative to variable selection when the sample size is small regarding to the dimension of covariates. Indeed, inverse regression achieves interesting coverage for reasonable time computation. Although our results are asymptotic, performances are challenging for finite sample and illustrates how this model can be used in practice. 

A future work could be the extension of this model to generalized linear model by considering other distributions of the noise of the inverse model. 

\bibliographystyle{plain}
\bibliography{biblio}

\appendix



\section{Details for the proofs}
\subsection{Notations for the proof \label{sec:notations}}
In this section, we introduce some notations useful for the proofs. 
\begin{enumerate}
\item Square root factorization of a positive definite matrix is denoted by $A^\frac12$ such as $A=(A^\frac12)^2$
\item The imaginary number $i$ is such as   $i^2=-1$
    \item Exponential of the trace of a matrix denoted by etr$(\cdot)$ returns the exponential of the sum of the diagonal terms of a matrix
    \item For $L\in \mathbb{N}^\star$, the multivariate gamma function is denoted by $\Gamma_L(\cdot)$ and defined as 
    \begin{align*}
        \Gamma_L(a) = \int_{A > 0} \text{etr}(-A)\text{det}(A)^{a-\frac12(L+1)}dA,
    \end{align*}
    where the real part of $a$ verifies $\text{Re}(a) > \frac12(L-1)$ and the integration space $A >0$ refers to $L \times L$ symmetric positive definite matrices
    \item Generalized Hayakawa polynomial introduced by \cite{crowther75} is denoted by $P_\kappa(\cdot,\cdot,\cdot)$ and defined for a complex matrix $T \in \mathcal M_{L,D}(\mathbb{C})$ and two real symmetric matrices $A \in \mathcal M_{D,D}(\mathbb{R})$ and $B \in \mathcal M_{L,L}(\mathbb{C})$ as 
    \begin{align*}
    P_\kappa(T,A,B) = \pi^{-\frac12DL}\int_U\text{etr}(-(U+iT)(U+iT)^T)C_\kappa(-BUAU^T)dU
    \end{align*}
        where $C_\kappa(S)$ is a zonal polynomial. For more details, we refer to \cite{GNbook}.
\end{enumerate}

\subsection{Proof of Theorem \ref{theoDistribGammaStar} \label{proofDistribGammaStar}}
\begin{proof}

We are interested in the distribution of
$$\widehat{\Gammab}^\star =\Sigmab+\widehat{\Ab}\Gammab  \widehat{\Ab}^{\top}.$$

From Proposition \ref{distribAchap}, we know that $\hat{\Ab} \sim \mathcal{MN}_{DL}(\Ab, \Sigmab, (\Yb^T\Yb)^{-1})$.

Remark that $\widehat{\Gammab}^\star$ is decomposed onto the sum of a diagonal matrix and a low rank matrix. This structure is general but involves a non invertible matrix.

First we focus on the distribution of $\widehat{\Ab}\Gammab  \widehat{\Ab}^{\top}$, where the randomness comes from $\widehat{\Ab}$. As $\widehat{\Ab} \in M_{D,L}(\mathbb{R})$ and $\Gammab \in M_{L,L}(\mathbb{R})$, we know that the $D \times D$ matrix $\widehat{\Ab}\Gammab  \widehat{\Ab}^{\top}$ is of rank $L$.
Then the distribution cannot be related to a Wishart distribution (arguments used in Section \ref{proofDistribSigmaStar} can not be used).



Finally, combining arguments on quadratic form developed in \cite{GNbook} and singular Wishart distributions introduced in \cite{uhlig1994}, we get the following density for $\Gammab^\star$ defined for symmetric definite positive matrices structured as the sum of a diagonal and a low rank matrix: 
\begin{align*}
& \pi^\frac{-DL+L^2}{2} \{ 2^{\frac12DL}\Gamma_L\left(\frac12L\right) \} ^{-1} \text{det}(\Sigmab)^{-\frac12L}\text{det}(\Bb)^{-\frac12D}\text{etr}\left( -\frac12 \Sigmab^{-1} \Ab \Yv^T\Yv\Ab^T\right) \nonumber \\
& \text{etr}\left( -\frac12q \Sigmab (\Gammab^\star - \Sigmab) \right) \left( \prod_{\lambda >0} \lambda(\Gammab^\star - \Sigmab)\right)^{\frac12(L-D-1)} \sum_{k=0}^\infty\sum_{\kappa} \frac{1}{(\frac12L)_\kappa k!} \nonumber \\
& P_\kappa\left( \frac{1}{\sqrt{2}}\Sigmab^{-\frac12} \Ab (\Yv^T\Yv)^{\frac12} (\mathbb I_L -q\Bb)^{-\frac12},\Bb^{-1}-q\mathbb I_L,\frac12 \Sigmab^{-\frac12} (\Gammab^\star - \Sigmab) \Sigmab^{-\frac12}\right) \nonumber
\end{align*}
 where $\lambda(A)$ corresponds to the eigenvalues of $A$ and $\Bb = (\Yv^T\Yv)^{-\frac12} \Gammab (\Yv^T\Yv)^{-\frac12} $, and $q > 0$ an arbitrary constant such that $\mathbb I_L -q\Bb$ is positive definite, and $\Gamma_L(\cdot)$, $\text{etr}(\cdot)$ and the Hayakawa polynomial $P_\kappa(\cdot,\cdot,\cdot)$ defined as in Section~\ref{sec:notations}.
\end{proof}

\subsection{Proof of Theorem \ref{theoDistribSigmaStar}\label{proofDistribSigmaStar}}

\begin{proof}
The purpose of this proof is to derive the distribution of $\widehat \Sigmab^\star = (\Gammab^{-1} + \widehat \Ab^T\Sigmab^{-1}\widehat \Ab)^{-1}$ knowing that $\widehat \Ab \sim \mathcal{MN}_{DL}\left(\Ab,\Sigmab,(\Yv^T\Yv)^{-1}\right)$ from Proposition \ref{distribAchap}. First, using Chapter 7 of \cite{GNbook}, we deduce that the quadratic form $\Sb_A = \widehat \Ab^T\Sigmab^{-1}\widehat\Ab$ has the following density: 
\begin{align}
& \{ 2^{\frac12DL}\Gamma_L\left(\frac12D\right) \} ^{-1} \text{det}((\Yv^\top\Yv)^{-1})^{-\frac12D}\text{det}(\Bb)^{-\frac12L}\text{etr}\left( -\frac12 (\Yv^\top\Yv) \Ab^\top \Sigmab^{-1}\Ab\right) \nonumber \\
& \text{etr}\left( -\frac12q \Yv^\top\Yv \Sb_A\right) \text{det}(\Sb_A)^{\frac12(D-L-1)} \sum_{k=0}^\infty\sum_{\kappa} \frac{1}{(\frac12D)_\kappa k!} \nonumber \\
& P_\kappa\left( \frac{1}{\sqrt{2}}(\Yv^\top\Yv)^\frac12 \Ab^\top \Sigmab^{-\frac12} (\mathbb I_D -q\Bb)^{-\frac12},\Bb^{-1}-q\mathbb I_D,\frac12 (\Yv^\top\Yv)^\frac12 \Sb_A (\Yv^\top\Yv)^\frac12\right) \nonumber
\end{align}
 defined for $\Sb_A > 0$, with $\Bb=\Sigmab^{\frac12}\Sigmab^{-1}\Sigmab^{\frac12} = \mathbb I_D$ as $\Sigmab$ is diagonal, $0 < q < 1$ an arbitrary constant, and $\Gamma_L(\cdot)$, $\text{etr}(\cdot)$ and the Hayakawa polynomial $P_\kappa(\cdot,\cdot,\cdot)$ defined as in Section~\ref{sec:notations}. Therefore the density of $\Sb_A$ simplifies
 \begin{align}
 & \{ 2^{\frac12DL}\Gamma_L\left(\frac12D\right) \} ^{-1} \text{det}(\Yv^\top\Yv)^{\frac12D}\text{etr}\left( -\frac12 (\Yv^\top\Yv) \Ab^\top \Sigmab^{-1}\Ab\right) \nonumber \\
& \text{etr}\left( -\frac12q \Yv^\top\Yv \Sb_A\right) \text{det}(\Sb_A)^{\frac12(D-L-1)} \sum_{k=0}^\infty\sum_{\kappa} \frac{1}{(\frac12D)_\kappa k!} \label{eq:density1} \\
& P_\kappa\left( \frac{(1-q)^{-\frac12}}{\sqrt{2}}(\Yv^\top\Yv)^\frac12 \Ab^\top \Sigmab^{-\frac12},(1-q)\mathbb I_D,\frac12 (\Yv^\top\Yv)^\frac12 \Sb_A (\Yv^\top\Yv)^\frac12\right) \nonumber
 \end{align}
 Remark that this density is related to a Wishart distribution, because we consider the quadratic form associated to a Gaussian random variable, the transformation being through a linear application $A$ leads to more complex formulae.
 
Then, transforming $\Tb = \Gammab^{-1}+\Sb_A$ we obtain the density of $(\Sigmab^\star)^{-1}$ as a function of $\Tb$ 
 \begin{align}
 f_{(\Sigmab^\star)^{-1}}(\mathbf{T}) =  f_{\Sb_A}(\Tb - \Gammab^{-1}), \text{defined for definite positive matrix } \Tb  \nonumber
  \end{align}
  where $f_{\Sb_A}$ refers to the density of $\Sb_A$ defined in Equation~\eqref{eq:density1}. Next, transforming $\Sigmab^\star = \Tb^{-1}$, with the Jacobian $J(\Tb \mapsto \Sigmab^\star) = \text{det}(\Sigmab^\star)^{L+1}$ we obtain the density of $\Sigmab^\star$
 \begin{align}
    f_{\Sigmab^\star}(\mathbf{T}) = f_{\Sb_A}((\Sigmab^\star)^{-1} - \Gammab^{-1})\text{det}(\Sigmab^\star)^{(L+1)}
 \end{align}
defined for positive definite matrix $\Sigmab^\star$, which gives the following final density using notation of Equation~\eqref{eq:density1}:
 \begin{align}
 & \{ 2^{\frac12DL}\Gamma_L\left(\frac12D\right) \} ^{-1} \text{det}(\Yv^\top\Yv)^{\frac12D}\text{det}(\Sigmab^\star)^{(L+1)} \nonumber \\
 &\text{etr}\left( -\frac12 (\Yv^\top\Yv) \Ab^\top \Sigmab^{-1}\Ab\right) \nonumber \text{etr}\left( -\frac12q \Yv^\top\Yv \left((\Sigmab^\star)^{-1} - \Gammab^{-1}\right)\right)  \\
 &\text{det}((\Sigmab^\star)^{-1} - \Gammab^{-1})^{\frac12(D-L-1)} \sum_{k=0}^\infty\sum_{\kappa} \frac{1}{(\frac12D)_\kappa k!} \nonumber \\
& P_\kappa\left( \frac{(1-q)^{-\frac12}}{\sqrt{2}}(\Yv^\top\Yv)^\frac12 \Ab \Sigmab^{-\frac12},(1-q)\mathbb I_D,\frac12 (\Yv^\top\Yv)^\frac12 ((\Sigmab^\star)^{-1} - \Gammab^{-1}) (\Yv^\top\Yv)^\frac12\right) \nonumber
 \end{align}
 with functions $\Gamma_L(\cdot)$ and $\text{etr}(\cdot)$ and the Hayakawa polynomial $P_\kappa(\cdot,\cdot,\cdot)$ defined as in Section~\ref{sec:notations}. 
 
  \end{proof}

\subsection{Proof of Lemma \ref{computeDiff}}
\label{proofLemma1}
\begin{proof}
We use the following lemma.
\begin{lemma}
If $\| A \| \leq 1$, then $(\mathbb I - \Ab)^{-1} = \mathbb I + \Ab + \Ab^2 + o(||\Ab||^2)$.
\end{lemma}

\begin{align*}
g(\Ab + h\Mb) - g(\Ab) & = h (\Gammab^{-1} + \Ab^\top \Sigmab^{-1} \Ab)^{-1} \Mb^\top \Sigmab^{-1}  \\
&- h (\Gammab^{-1} + \Ab^\top \Sigmab^{-1} \Ab)^{-1} (\Mb^\top\Sigmab^{-1}\Ab + \Ab^\top\Sigmab^{-1}\Mb) \\
& \times (\Gammab^{-1} + \Ab^\top \Sigmab^{-1} \Ab)^{-1} \Ab \Sigmab^{-1}\\
& + O(h^2)\\
Dg(\Ab) . M &=  (\Gammab^{-1} + \Ab^\top \Sigmab^{-1} \Ab)^{-1} \Mb^\top \Sigmab^{-1}  \\
&- (\Gammab^{-1} + \Ab^\top \Sigmab^{-1} \Ab)^{-1} (\Mb^\top\Sigmab^{-1}\Ab + \Ab^\top\Sigmab^{-1}\Mb) \\
& \times  (\Gammab^{-1} + \Ab^\top \Sigmab^{-1} \Ab)^{-1} \Ab \Sigmab^{-1}
\end{align*}
Next, remember that $h\Mb = (\widehat \Ab - \Ab)$, we have :
\begin{align}
Dg(\Ab) . (\widehat \Ab - \Ab) &=  (\Gammab^{-1} + \Ab^\top \Sigmab^{-1} \Ab)^{-1} [\widehat \Ab - \Ab]^\top \Sigmab^{-1} \label{eq:DFA} \nonumber \\
&- (\Gammab^{-1} + \Ab^\top \Sigmab^{-1} \Ab)^{-1} \left([\widehat \Ab - \Ab]^\top\Sigmab^{-1}\Ab + \Ab^\top\Sigmab^{-1}[\widehat \Ab - \Ab]\right) \nonumber \\ 
& \times (\Gammab^{-1} + \Ab^\top \Sigmab^{-1} \Ab)^{-1} \Ab \Sigmab^{-1} \nonumber\\
\end{align}

Then, we compute the covariance.
We  decompose  it as the following.
\begin{align*}
Cov(\text{vec}(Dg(\Ab) . (\widehat \Ab - \Ab))) &= 
var(\text{vec}(\Sigmab^\star (\widehat \Ab - \Ab)^\top\Sigmab^{-1})) \\
& + var(\text{vec}( \Sigmab^\star (\widehat \Ab - \Ab)^\top\Sigmab^{-1} \Ab \Ab^\star)) \\
&+ var(\text{vec}(\Ab^\star (\widehat \Ab - \Ab) \Ab^\star))\\
& -2 cov(\text{vec}(\Sigmab^\star (\widehat \Ab - \Ab)^\top\Sigmab^{-1}), \text{vec}(\Sigmab^\star (\widehat \Ab - \Ab)^\top\Sigmab^{-1} \Ab \Ab^\star)) \\
&-2 cov (\text{vec}(\Sigmab^\star (\widehat \Ab - \Ab)^\top\Sigmab^{-1}), \text{vec}(\Ab^\star (\widehat \Ab - \Ab) \Ab^\star))\\
&-2 cov (\text{vec}(\Sigmab^\star (\widehat \Ab - \Ab)^\top\Sigmab^{-1} \Ab \Ab^\star), \text{vec}(\Ab^\star (\widehat \Ab - \Ab) \Ab^\star))
\end{align*}
Then, we want to compute each term explicitly.
\begin{align*}
var(\text{vec}(\Sigmab^\star (\widehat \Ab - \Ab)^\top\Sigmab^{-1})) &= (\Sigmab^{-1} \otimes \Sigmab^\star \Gammab \Sigmab^\star)\\
 var(\text{vec}( \Sigmab^\star (\widehat \Ab - \Ab)^\top\Sigmab^{-1} \Ab \Ab^\star))&= ((\Ab^\star)^\top \Ab^\top \Sigmab^{-1} \Ab \Ab^\star \otimes \Sigmab^\star \Gammab \Sigmab^\star)\\
var(\text{vec}(\Ab^\star (\widehat \Ab - \Ab) \Ab^\star))&= ((\Ab^\star)^\top \Gammab \Ab^\star \otimes \Ab^\star \Sigmab (\Ab^\star)^\top )\\
cov(\text{vec}(\Sigmab^\star (\widehat \Ab - \Ab)^\top\Sigmab^{-1}), \text{vec}(\Sigmab^\star (\widehat \Ab - \Ab)^\top\Sigmab^{-1} \Ab \Ab^\star))&= (\Sigmab^{-1} \Ab \Ab^\star  \otimes \Sigmab^\star \Gammab \Sigmab^\star)\\
cov (\text{vec}(\Sigmab^\star (\widehat \Ab - \Ab)^\top\Sigmab^{-1}), \text{vec}(\Ab^\star (\widehat \Ab - \Ab) \Ab^\star))&= 
(\mathbf{I} \otimes \Sigmab^\star \Gammab) T^{-1}_{LD}((\Ab^\star)^\top \otimes \Ab^\star)
\\
cov (\text{vec}(\Sigmab^\star (\widehat \Ab - \Ab)^\top\Sigmab^{-1} \Ab \Ab^\star), \text{vec}(\Ab^\star (\widehat \Ab - \Ab) \Ab^\star)) &= ((\Ab^\star)^\top \Ab^\top \otimes \Sigmab^\star \Gammab)  T^{-1}_{LD}ÃÂ  ((\Ab^\star)^\top \otimes \Ab^\star)
\end{align*}

Putting everything together, we get the following.

\begin{align*}
Cov(\text{vec}(Dg(\Ab) . (\widehat \Ab - \Ab))) &= 
(\Sigmab^{-1} \otimes \Sigmab^\star \Gammab \Sigmab^\star) +
((\Ab^\star)^\top \Ab^\top \Sigmab^{-1} \Ab \Ab^\star \otimes \Sigmab^\star \Gammab \Sigmab^\star)\\ &+
((\Ab^\star)^\top \Gammab \Ab^\star \otimes \Ab^\star \Sigmab (\Ab^\star)^\top ) - 2 
(\Sigmab^{-1} \Ab \Ab^\star  \otimes \Sigmab^\star \Gammab \Sigmab^\star) \\&-2
(\mathbf{I} \otimes \Sigmab^\star \Gammab) T^{-1}_{LD}((\Ab^\star)^\top \otimes \Ab^\star) \\& -2
((\Ab^\star)^\top \Ab^\top \otimes \Sigmab^\star \Gammab)  T^{-1}_{LD}ÃÂ  ((\Ab^\star)^\top \otimes \Ab^\star) \\
&=\left((\Sigmab^{-1} + (\Ab^\star)^\top \Ab^\top \Sigmab^{-1} \Ab \Ab^\star -2 \Sigmab^{-1} \Ab \Ab^\star)\otimes \Sigmab^\star \Gammab \Sigmab^\star\right) \\
&+ ((\Ab^\star)^\top \Gammab \Ab^\star \otimes \Ab^\star \Sigmab (\Ab^\star)^\top ) \\
&-2 \left(
(\mathbf{I} \otimes \Sigmab^\star \Gammab)+
((\Ab^\star)^\top \Ab^\top \otimes \Sigmab^\star \Gammab) \right)  T^{-1}_{LD}ÃÂ  ((\Ab^\star)^\top \otimes \Ab^\star)
\end{align*}

\end{proof}

\section{Computation for univariate response - easier to understand}
Whereas the method becomes less interesting for $L=1$, because we reduce the problem to 1 dimension through the inversion method, we detail here the theoretical result for the scalar response case as computations are easier to derive and to understand. The only goal of this section is then to be pedagogical.

\label{L=1}

\subsection[Asymptotic normality of regression coefficients]{Asymptotic normality of $\widehat{\Ab}^\star$}

When we consider a real response, the $\Delta$-method is used to deduce the distribution of $\widehat{\Ab}^\star \in \mathbb{R}^D$ from the distribution of $\widehat{\Ab} \in \mathbb{R}^D$.
To highlight the univariate response, we denote $\gamma = \Gammab \in \mathbb{R}$ and $s^\star = \Sigmab^\star \in \mathbb{R}$.
First, let recall the distribution of the least square estimator.

\begin{prop}[Distribution of $\widehat{\Ab}$]
\label{distribA}
Suppose $((\Xb_1,\Yb_1),\ldots, (\Xb_N,\Yb_N))$ is a sequence of $iid$ random variables satisfying the inverse regression model defined in Equations \eqref{eq:model::Y} and \eqref{eq:model::X|Y}.
Then, the following holds for the estimator defined in Equation \eqref{AChap}.
\begin{align*}
\widehat{\Ab}ÃÂÃÂÃÂÃÂ &\sim \mathcal{N}_D(\Ab, \Sigmab (\Yv^\top \Yv)^{-1}).
\end{align*}
\end{prop}

In Proposition \ref{propNablag}, we define the function $g : \Ab \mapsto \Ab^\star$ and compute its gradient.

\begin{prop}
\label{propNablag}
Let 
\begin{align*}
g : \mathbb{R}^{D \times 1} &\rightarrow \mathbb{R}^{1 \times D} \\
 \Ab &\mapsto \Ab^\star =  s^\star\Ab^\top\Sigmab^{-1} = (\Gammab^{-1} + \Ab^\top\Sigmab^{-1}\Ab)^{-1}\Ab^\top\Sigmab^{-1}
\end{align*}
Then, 
$$
\nabla g(\Ab) = - 2s^\star\Sigmab^{-1}\Ab\Ab^\star + s^\star(\partial \Ab)^\top\Sigmab^{-1}
$$
where $(\partial \Ab)=\mathbb I_D$ is the differentiation of $\Ab$.
\end{prop}
\begin{proof}
As $g$ is the product of $s^\star$ and $\Ab^\top\Sigmab^{-1}$, we have 
$$
\nabla g = \partial(s^\star)\Ab^\top\Sigmab^{-1} + s^\star \partial(\Ab^\top\Sigmab^{-1}) $$

As $L=1$, we have $\partial A=\mathbf I_D$.

 Next, we need to compute $\partial(\Ab^\top\Sigmab^{-1})$:
$$
\partial(\Ab^\top\Sigmab^{-1}) = \partial(\Ab)^\top \Sigmab^{-1}
$$

Finally, we  compute $\partial(s^\star)$: as $ \partial(\gamma^{-1} + \Ab^\top\Sigmab^{-1}\Ab ) = 2\Sigmab^{-1}\Ab$,
$$
 \partial(s^\star) =  - s^\star \partial(\gamma^{-1} + \Ab^\top\Sigmab^{-1}\Ab ) s^\star.
$$

\end{proof}

\begin{theorem}[Asymptotic normality of $\widehat{\Ab}^\star$] 
\label{asympNormalityL1}
Suppose $((\Xb_1,\Yb_1),\ldots, (\Xb_N,\Yb_N))$ is a sequence of $iid$ random variables  satisfying the  model defined in Equations \eqref{eq:model::X} and \eqref{eq:model::Y|X}.
Let 
\begin{align*}
g : \mathbb{R}^{D \times 1} &\rightarrow \mathbb{R}^{1 \times D} \\
 \Ab &\mapsto \Ab^\star =  s^\star\Ab^\top\Sigmab^{-1} = (\Gammab^{-1} + \Ab^\top\Sigmab^{-1}\Ab)^{-1}\Ab^\top\Sigmab^{-1}
\end{align*}
and let $\nabla g(\Ab)$  the gradient of  $g$ evaluated in a matrix $\Ab$.
Then, the following holds for the estimator $\widehat{\Ab}^\star = g(\widehat{A})$ defined in Equation \eqref{AChapStar}.
\begin{align}
\sqrt{N} (\widehat \Ab^\star -  \Ab^\star) \underset{N \rightarrow +\infty}{\rightarrow} \mathcal N_D(0, \Gammab^{-1} \nabla g( \Ab)^\top  \Sigmab \nabla g( \Ab)).\label{asympNormalityAChapStarL1}
\end{align}
Moreover, let $P$ be the Cholesky decomposition of $\Sigmab$: $\Sigmab = P^\top P$. Then, 
\begin{align}
\sqrt{N} \sqrt{\gamma} (\nabla g( \widehat\Ab) P)^{-1} (\widehat \Ab^\star -  \Ab^\star) \underset{N \rightarrow +\infty}{\rightarrow} \mathcal N_D(0, \mathbf{I}).\label{asympNormalityAChapStarL1Slustky}
\end{align}

\end{theorem}
\begin{proof}
Equation \eqref{asympNormalityAChapStarL1} relies on the $\Delta$-method applied to $\widehat{\Ab}$ which is Gaussian, as detailed in Proposition \ref{distribA}, through the function $g$ defined in Proposition \ref{propNablag}.

Equation \eqref{asympNormalityAChapStarL1Slustky} relies on Slutsky Lemma, which is used because $\widehat\Ab$ converges in probability to $\Ab$. 
\end{proof}

\subsection{Confidence region for $\widehat{\Ab^\star}$}
From Theorem \ref{asympNormalityL1},  confidence regions for $\widehat{\Ab}^\star$ are deduced using the following lemma, which makes the link between $\chi^2$ distribution and multivariate Gaussian distribution.
\begin{lemma}
\label{chi2}
If $X \sim \mathcal N_k(\mu,\Sigma)$ with $\Sigma$ known, then a confidence region for $\mu$ at level $1-\alpha$ is $\mathcal{R}_{\mu,\alpha}$, with
\begin{align*}
\mathcal{R}_{\mu, \alpha} &= \left\{ x \in \mathbb R^k \text{ such that }(x-\widehat\mu)^\top\Sigma^{-1}(x-\widehat\mu) \leq \chi^2_k(1-\alpha) \right\}
\end{align*}
where $\chi^2_D (1-\alpha)$ is the $(1-\alpha)$ quantile of the $\chi^2$ distribution with $k$ degrees of freedom.

If $X \sim \mathcal N_k(\mu,\Sigma)$ with $\Sigma$ unknown, then a confidence region for $\mu$ at level $1-\alpha$ is $\tilde{\mathcal{R}}_{\mu,\alpha}$, with
\begin{align*}
\tilde{\mathcal{R}}_{\mu, \alpha} &= \left\{ x \in \mathbb R^k \text{ such that } n(x-\widehat\mu)^\top S ^{-1}(x-\widehat\mu) \leq T^2_{k,n-1}(1-\alpha) \right\}
\end{align*}
with $S = 1/(n-1) \sum_i (x_i - \bar{x})(x_i - \bar{x})^\top$ the empirical covariance, and $T^2_{k,n-1}(1-\alpha)$ the quantile of the Hotelling's $T^2$ distribution with parameters $k$ and $n-1$.
\end{lemma}

Then, we can construct an asymptotic confidence region for $\Ab^\star$ with level $1-\alpha$.
Remark that combining Slutsky's lemma and Lemma \ref{chi2} leads to a $\chi^2$ distribution when the covariance is estimated as done for $\Ab^\star$.
\begin{theorem}[Confidence region for $\Ab^\star$]
Suppose $((\Xb_1,\Yb_1),\ldots, (\Xb_N,\Yb_N))$ is a sequence of $iid$ random variables  satisfying the model defined in Equations \eqref{eq:model::X} and \eqref{eq:model::Y|X}.
Then, 
$$
 P\left( \Ab^\star \in \tilde{\mathcal{R}}_{\Ab^\star, \alpha} \right) \underset{n \rightarrow +\infty}{\rightarrow} 1-\alpha
 $$
where 
$$\tilde{\mathcal{R}}_{\Ab^\star, \alpha} = \left\{ \ab^\star \in \mathbb R^D \text{ s.t. } \gamma (\ab^\star-\widehat\Ab^\star)^\top ( \nabla g( \widehat\Ab)^\top  \Sigmab \nabla g(\widehat \Ab))^{-1} (\ab^\star-\widehat\Ab^\star) 
 \leq \chi^2_{D}(1-\alpha) \right\}.
 $$

\end{theorem}
%

\subsection{Prediction region}
For a new profile $\xb_{N+1}$, the prediction is get by $\widehat\yb_{N+1} = \widehat{\Ab}^\star \xb_{N+1}$ as described in Section \ref{section::prediction}.
A prediction region is then deduced in the following theorem.

\begin{theorem}[Prediction region for $\yb$]
Suppose $((\Xb_1,\Yb_1),\ldots, (\Xb_N,\Yb_N))$ is a sequence of $iid$ random variables  satisfying the model defined in Equations \eqref{eq:model::X} and \eqref{eq:model::Y|X}.
Then, 
$$
 P\left( \yb_{n+1} \in \widetilde{\mathcal{PR}}_{\yb, \alpha} \right) \underset{n \rightarrow +\infty}{\rightarrow} 1-\alpha$$
where 
$$\widetilde{\mathcal{PR}}_{\yb, \alpha} = \left\{ y\in \mathbb R \text{ s.t. }   (y-\widehat\Ab^\star \xb_{N+1})^\top v^{-1} (y-\widehat\Ab^\star\xb_{N+1}) 
 \leq \chi^2_{D}(1-\alpha)
 \right\}
$$
with $v =  \gamma \xb_{N+1}^\top\nabla g( \widehat\Ab)^\top  \Sigmab \nabla g(\widehat \Ab) \xb_{N+1} + s^\star$.
\end{theorem}








 \end{document}

%% file: preamble.tex
\def\XS{\xspace}
\DeclareMathAlphabet{\mathb}{OML}{cmm}{b}{it}
\def\sbm#1{\ensuremath{\mathb{#1}}}                
\def\sbmm#1{\ensuremath{\boldsymbol{#1}}}          

\def\Ab{{\sbm{A}}\XS}  \def\ab{{\sbm{a}}\XS}
\def\Bb{{\sbm{B}}\XS}

\def\Mb{{\sbm{M}}\XS}

\def\Sb{{\sbm{S}}\XS}  
\def\Tb{{\sbm{T}}\XS}

\def\Xb{{\sbm{X}}\XS}  \def\xb{{\sbm{x}}\XS}
\def\Yb{{\sbm{Y}}\XS}  \def\yb{{\sbm{y}}\XS}

\def\Gammab    {{\sbmm{\Gamma}}\XS}

      \def\Sigmab    {{\sbmm{\Sigma}}\XS}

\def\PM{\kern0pt^{\textrm{{\scriptsize PM}}}\kern0pt}
\def\MMAP{\kern1pt^{\textrm{{\tiny MMAP}}}\kern-1pt}

\newcommand{\Xv}{\ensuremath{\mathbf{X}}} 
\newcommand{\xv}{\ensuremath{\mathbf{x}}} 
\newcommand{\Yv}{\ensuremath{\mathbf{Y}}} 

\def\wp1{\mathrm{w.p.} 1}  